\pdfoutput=1
\RequirePackage{ifpdf}
\ifpdf 
\documentclass[pdftex]{sigma}
\else
\documentclass{sigma}
\fi

\numberwithin{equation}{section}

\newtheorem{thm}{Theorem}[section]

\newtheorem{corl}[thm]{Corollary}
 { \theoremstyle{definition}
\newtheorem{defn}[thm]{Definition}
\newtheorem{Example}[thm]{Example}
\newtheorem{rmk}[thm]{Remark}}

\DeclareMathOperator{\End}{End} 
\DeclareMathOperator{\Id}{Id} 
\DeclareMathOperator{\linspan}{span} 
\DeclareMathOperator{\Mult}{Mult} 

\newcommand{\rInd}{\operatorname{\mbox{$r$-$\operatorname{Ind}$}}}

\newcommand{\C}{\mathbb{C}} 
\renewcommand{\H}{\mathcal{H}} 
\renewcommand{\S}{\mathcal{S}}
\renewcommand{\L}{\mathcal{L}}
\newcommand{\K}{\mathcal{K}} 
\newcommand{\N}{\mathbb{N}} 
\newcommand{\ox}{\otimes} 
\newcommand{\R}{\mathbb{R}} 
\newcommand{\ol}[1]{\overline{#1}} 

\begin{document}

\allowdisplaybreaks

\newcommand{\arXivNumber}{1612.03559}

\renewcommand{\PaperNumber}{041}

\FirstPageHeading

\ShortArticleName{Non-Commutative Vector Bundles for Non-Unital Algebras}

\ArticleName{Non-Commutative Vector Bundles\\ for Non-Unital Algebras}

\Author{Adam RENNIE and Aidan SIMS}
\AuthorNameForHeading{A.~Rennie and A.~Sims}
\Address{School of Mathematics and Applied Statistics,\\
University of Wollongong, Northfields Ave 2522, Australia}
\Email{\href{mailto:renniea@uow.edu.au}{renniea@uow.edu.au}, \href{mailto:asims@uow.edu.au}{asims@uow.edu.au}}

\ArticleDates{Received December 13, 2016, in f\/inal form June 12, 2017; Published online June 16, 2017}

\Abstract{We revisit the characterisation of modules over non-unital $C^*$-algebras analogous to
modules of sections of vector bundles. A fullness condition on the associated multiplier
module characterises a class of modules which closely mirror the commutative case. We
also investigate the multiplier-module construction in the context of bi-Hilbertian
bimodules, particularly those of f\/inite numerical index and f\/inite Watatani index.}


\Classification{57R22; 46L85}

\section{Introduction}\label{sec:intro}

The Serre--Swan theorem says that the Hilbert modules over unital commutative $C^*$-algebras that can be realised as the modules of sections of locally trivial vector bundles over compact spaces are precisely the f\/inite projective modules. By direct analogy, we \emph{define} non-commutative vector bundles over unital $C^*$-algebras to be f\/inite projective modules, this def\/inition being justif\/ied by the connection to the non-commutative def\/inition of $K$-theory.

This note revisits the question of the correct notion of a non-commutative vector bundle over a non-unital $C^*$-algebra. We prove the equivalence of several conditions on a~Hilbert module $E$ over a non-unital $C^*$-algebra $A$ that, when $A$ is commutative, characterise modules of sections of vector bundles. This extends previous partial characterisations \cite{RennieSmooth}. In particular, our results apply to suspensions of f\/initely generated projective modules over unital $C^*$-algebras. This is an important motivation since it allows us to apply techniques like those of \cite{RRSext} to study Cuntz--Pimsner algebras of suspended $C^*$-correspondences, and thereby to bootstrap computational techniques from even $K$-groups and Kasparov groups to their odd counterparts.

We then investigate the structure of multiplier modules associated to bimodules that are bi-Hilbertian in the sense of \cite{KajPinWat}. We prove that, under mild hypotheses, the bi-Hilbertian structure and f\/inite numerical index pass to the multiplier module. We also establish that, by contrast, the multiplier module frequently does not have f\/inite Watatani index. Indeed the multiplier module has f\/inite Watatani index if and only if it is f\/initely generated and projective as a right-Hilbert module over the multiplier algebra, which in turn holds if and only if it is full as a left-Hilbert module over the multiplier algebra.

We start by brief\/ly recalling what is already known. Then we prove our f\/irst result, Theo\-rem~\ref{thm:equiv-one}, which shows that if $E$ is a Hilbert module over a~$\sigma$-unital $C^*$-algebra $A$, then its multiplier module is full as a left-Hilbert $\End_A(E)$-module if and only if $E$ is a f\/initely generated projective module over a~suitable unitisation of $A$. We discuss some consequences of this result. In particular, by applying our results to the setting of commutative $C^*$-algebras, we prove that every locally trivial vector bundle $V$ over a locally compact space $X$ of f\/inite topological dimension extends to a locally trivial vector bundle $V^c$ over some compactif\/ication $X^c$ of~$X$. The compactif\/ication $X^c$ required, and the isomorphism class of the extension $V^c$, depend on a~choice of frame for~$V$, which we illustrate by example.

We then recall the notion of a bi-Hilbertian bimodule \cite{KajPinWat}. We prove that if $E$ is countably generated with injective left action and f\/inite Watatani index, then the bi-Hilbertian structure, and also f\/inite numerical index, pass from $E$ to its multiplier module in a natural way. We then describe a number of conditions that are equivalent to the multiplier module being f\/initely generated and projective, and show how this applies to modules over commutative $C^*$-algebras.

\section{Finite projective modules and non-unital analogues}
\label{sec:fgps} Throughout the paper, $A$ denotes a $\sigma$-unital $C^*$-algebra. Given a right $C^*$-$A$-module $E$, we denote the $C^*$-algebra of adjointable operators on $E$ by $\End_A(E)$. For $e,\,f\in E$, the \mbox{rank-1} endomorphism $g \mapsto e \cdot (f \,|\, g)_A$ is denoted by $\Theta_{e,f}$, and is adjointable with adjoint $\Theta_{f,e}$. We write $\End_A^0(E)$ for $\overline{\linspan}\{\Theta_{e,f} \colon e,f \in E\}$, the closed $*$-ideal of compact endomorphisms in $\End_A(E)$. We write $\ell^2(A)$ for the standard $C^*$-module over $A$; that is $\ell^2(A)$ is equal to $\big\{\xi \colon \N \to A \,|\, \sum\limits^\infty_{n=1} \xi_i^*\xi_i\text{ converges in }A\big\}$ with inner product $(\xi \,|\, \eta)_A = \sum_i \xi^*_i \eta_i$. We say that a~(right) inner product module $E$ over a~$C^*$-algebra $A$ is \emph{full} if the closed span of the inner products $(e\,|\, f)_A$ is all of~$A$.

In keeping with the preceding paragraph, throughout the paper, by a \emph{vector bundle} over a~locally compact Hausdorf\/f space $X$, we will mean a~locally trivial f\/inite-rank complex vector bundle equipped with a continuous family of inner products. If $X$ is paracompact, then every f\/inite rank vector bundle over $X$ admits such a family of inner products.

A \emph{frame} for a right $C^*$-$A$-module $E$ is a sequence $\{e_j\}_{j\geq 1}\subset E$ such that
\begin{gather*}
\sum_{j\geq 1}\Theta_{e_j,e_j}\text{ converges strictly to }{\rm Id}_E.
\end{gather*}
If $\{e_j\}_{j \ge 1}$ is a frame for $E$, then $E$ is generated as a right $A$-module by the $e_j$, so it is countably generated.

Any frame $\{e_j\}$ for $E$ determines a stabilisation map in the sense of Kasparov: there is an adjointable map $v \colon E \to \ell^2(A)$ such that
\begin{gather*}
v(e) = \big((e_j \,|\, e)_A\big)_{j\geq 1}\quad\text{ for all $e \in E$.}
\end{gather*}
We have $v^*v={\rm Id}_E$, and so $p := vv^*$ is a projection in $\End_A(\ell^2(A))$; specif\/ically, $p$ satisf\/ies
\begin{gather*}
(p \xi)_i = \sum_j (e_i \,|\, e_j)_A \xi_j \qquad\text{for all $\xi \in \ell^2(A)$.}
\end{gather*}

One of the fundamental points of contact between non-commutative geometry and classical geometry is the celebrated Serre--Swan theorem. Given a~vector bundle $V \to X$ over a~compact space $X$, we write $\Gamma(X, V)$ for the space of continuous sections of~$V$.

\begin{thm}[\cite{Swan}]\label{thm:SS} Let $X$ be a compact Hausdorff space. A $($right$)$ $C(X)$-module $M$ is finitely generated and projective if and only if there is a vector bundle $V\to X$ such that $M\cong \Gamma(X,V)$.
\end{thm}

We can remove the word ``projective'' from the statement of Theorem~\ref{thm:SS} if we instead consider full inner-product modules over $C(X)$~-- or more generally inner-product modules $E$ such that $(E \,|\, E)_{C(X)}$ is a unital, and hence complemented, ideal of~$C(X)$. This is because every full f\/initely generated right inner product $C(X)$-module can be made into a $C^*$-module~$M$ which, by Kasparov's stabilisation theorem, is automatically projective.

When $X$ is non-compact, the natural $C_0(X)$-module arising from a vector bundle $V \to X$ is the module $\Gamma_0(X,V)$ of continuous sections of $V$ such that $x \mapsto (\xi(x) \,|\, \xi(x))$ belongs to~$C_0(X)$. Important examples of this are restrictions $V \to X$ of vector bundles $V^c\to X^c$ over some compactif\/ication $X^c$ of $X$. The following non-unital Serre--Swan theorem characterises the algebraic structure of such modules.

\begin{thm}[{\cite[Theorem~8]{RennieSmooth}}]
\label{thm:nonunital-SS} Let $X$ be a locally compact Hausdorff space and $X^c$ a~compacti\-fication of $X$. Set $A=C_0(X)$ and $A_b=C(X^c)$. A right $A$-module $E$ is of the form $pA^n$ for some projection $p\in M_n(A_b)$ if and only if there is a $($locally trivial$)$ vector bundle $V \to X^c$ such that $E \cong \Gamma_0(X,V|_X)$.
\end{thm}

Theorem~\ref{thm:nonunital-SS} provides a reasonable algebraic characterisation of modules of the form $\Gamma_0(X, V)$ where the vector bundle $V$ extends to a bundle over some compactif\/ication of $X$. We use this to motivate our def\/inition of the non-commutative analogue.

Recall that a \emph{unitisation} of a nonunital $C^*$-algebra $A$ is an embedding $\iota\colon A \hookrightarrow A_b$ of $A$ as an essential ideal of a~unital $C^*$-algebra $A_b$.

\begin{defn} \label{def:a-b-fgp} Let $A$ be a nonunital separable $C^*$-algebra and $\iota\colon A\hookrightarrow A_b$ a~unitisation of $A$. An \emph{$A_b$-finite projective $A$-module} is a right $A$-module that is isomorphic to $pA^n$ for some $n \in \N$ and some projection $p\in M_n(A_b)$.
\end{defn}

To state our main result, we need to recall how to embed a Hilbert $A$-module $E$ in the associated multiplier module, a tool used frequently in the Gabor-analysis literature (see, for instance, \cite{AB,RT}).

\begin{defn} \label{defn:mult-link} Given a right $C^*$-$A$-module $E$, recall that the \emph{linking algebra} $\L(E)$ is the algebra
\begin{gather*}
\L(E)=\End^0_A(E\oplus A)
=\left\{\begin{pmatrix} \,T & e\\ \ol{f} & a\end{pmatrix}\colon T\in \End_A^0(E),\ a\in A,\ e\in E,\ \ol{f}\in \ol{E}\right\},
\end{gather*}
where $\ol{E}$ is the conjugate module (a left $A$-module). Let $r={\rm Id}_E\oplus 0\in\Mult(\L(E))$ and $s=0\oplus 1_{\Mult(A)}\in \Mult(\L(E))$ and def\/ine
\begin{gather*}
\Mult(E)=r\Mult(\L(E))s.
\end{gather*}
Then $\Mult(E)$ is a right $C^*$-$\Mult(A)$-module with right action implemented by right multiplication in $\Mult(\L(E))$, and right inner-product $(e \,|\, f)_{\Mult(A)} = e^* f$.
\end{defn}

Writing $\operatorname{Hom}_A(A, E)$ for the Banach space of adjointable operators from $A$ to $E$, we then have $\Mult(E) \cong \operatorname{Hom}_A(A, E)$: for any $m \in \Mult(E)$, the map $a \mapsto m a$ belongs to $\operatorname{Hom}_A(A, E)$, and conversely if $m \in \operatorname{Hom}_A(A, E)$, then there is a multiplier of $\L(E)$ such that
\begin{gather*}
m\begin{pmatrix} T & e\\ \ol{f} & a\end{pmatrix} = \begin{pmatrix} m\circ (f \,|\, \cdot)_A & ma \\ 0 & 0\end{pmatrix}.
\end{gather*}

For $e,f \in \Mult(E)$, let $\Theta_{e,f}$ be the associated rank-one operator on $\Mult(E)$. We show that $E \subseteq \Mult(E)$ is invariant for $\Theta_{e,f}$. To see this, f\/ix $g \in E$, use Cohen factorisation to write $g = g' \cdot a$ for some $g' \in E$ and $a \in A$. Regarding $g'$ as an element of $\Mult(E)$, we calculate
\begin{gather}\label{eq:E multE inv}
\Theta_{e,f}(g) = e \cdot (f \,|\, g'\cdot a) = e\cdot (f \,|\, g)a \in \Mult(E)\cdot A,
\end{gather}
which is equal to $E$ by \cite[Remark~3.2(a)]{AB}.

We deduce that $\Theta_{e,f}|_E$ is adjointable with adjoint $\Theta_{f,e}|_E$. Hence $\Theta_{e,f}|_E \in \End_A(E)$. Abusing notation a little, we will just write $\Theta_{e,f}$ for this operator on $E$ henceforth. The left inner product ${}_{\End_{\Mult(A)}^0(\Mult(E))}(e\,|\, f)=\Theta_{e,f}$
gives a norm on $\Mult(E)$ equivalent to that coming from the right $\Mult(A)$-valued inner product. This can be seen directly from the linking algebra picture above.

The inclusion $\End_{\Mult(A)}^0(\Mult(E))\hookrightarrow \End_A(E)$ is injective, and so we can regard $\Mult(E)$ as a left $\End_A(E)$-$C^*$-module. We will denote the left inner product by ${}_{\End_A(E)}(\cdot\,|\,\cdot)$; we then have ${}_{\End_A(E)}(e \,|\, f)
= \Theta_{e,f}|_E$ for $e,f \in \Mult(E)$.

\section{Finite projectivity and the multiplier module}

\begin{thm} \label{thm:equiv-one}
Let $A$ be a non-unital $\sigma$-unital $C^*$-algebra and let $E$ be a right $C^*$-$A$-module. The following are equivalent.
\begin{enumerate}\itemsep=0pt
\item[$1.$] 
There is a finite subset $F \subseteq \Mult(A)$ such that, putting
 $A_b := C^*(F \cup A) \subseteq \Mult(A)$, the module $E$ is $A_b$-finite
 projective.
\item[$2.$] 
The module $E$ is $\Mult(A)$-finite projective.
\item[$3.$] 
The module $\Mult(E)$ is finitely generated and projective as a right $\Mult(A)$ module.
\item[$4.$] 
The module $\Mult(E)$ is full as a left $C^*$-module over $\End_A(E)$.
\end{enumerate}
\end{thm}
\begin{proof}
The universal property of $\Mult(A)$ shows that for any unitisation $A_b$ of $A$ as in~(1), we have an inclusion $A_b\hookrightarrow\Mult(A)$. Thus $pA_b^n\ox_{A_b}\Mult(A)\cong p\Mult(A)^n$ and so~(1) implies~(2).

For~(2) implies~(1), we observe that since $p \in M_n(\Mult(A))$, it has f\/initely many matrix entries $(p_{i,j})^n_{i=1}$. Let $F = \{p_{i,j} \colon i,j \le n\} \cup \{1_{\Mult(A)}\}$, and let $A_b = C^*(A \cup F) \subseteq \Mult(A)$. Then $A_b$ is a~unitisation of $A$, $p \in M_n(A_b)$, and $E \cong p A^n$ as a right $A$-module as required.

For~(2) implies~(3), recall that we may identify $\Mult(E)$ with $\operatorname{Hom}_A(A,E)$, from which we deduce that $\Mult(pA^n)\cong p\Mult(A)^n$.

To see that~(3) implies~(2), f\/ix a f\/inite frame $\{\xi_j\}$ for $\Mult(E)$ as a right-Hilbert $\Mult(A)$-module. By Kasparov's stabilisation theorem applied to this frame, there exist an integer $n \ge 1$, a projection $p \in M_n(\Mult(A))$, and an isomorphism $\rho \colon \Mult(E) \to p\Mult(A)^n$ of right $\Mult(A)$-modules. We claim that $\rho(E) = pA^n$.

To see this, f\/irst f\/ix $e \in E$. Apply the strong form \cite[Proposition~2.31]{tfb} of Cohen factorisation to write $e = e' \cdot (e' \,|\, e')$ for some $e' \in E$. Then $\rho(e) = p \rho(e) = p \rho(e') \cdot(e' \,|\, e')$. Since $(e' \,|\, e') \in A$ and since $\Mult(E) \cdot A = E$, we deduce that $\rho(e) \in p A^n$. So $\rho(E) \subseteq pA^n$. For the reverse inclusion, we f\/ix $a \in p A^n$, and use Cohen factorisation in $p
A^n$ to write $a = a' \cdot a''$ for some $a' \in pA^n$ and $a'' \in A$. Write $a' = \rho(\xi)$ for some $\xi \in \Mult(E)$. Using again that $\Mult(E)\cdot A \subseteq E$, we see that $a = \rho(\xi)\cdot a'' = \rho(\xi \cdot a'') \in \rho(E)$.

To see that~(3) implies~(4), suppose that $\Mult(E)$ is f\/initely generated and projective as a~right $\Mult(A)$-module. Then $\End_{\Mult(A)}(\Mult(E)) = \End_{\Mult(A)}^0(\Mult(E))$, and so $\End_A(E) \subseteq \End_{\Mult(A)}(\Mult(E))$ belongs to the range of the outer product. So $\Mult(E)$ is full as a left $C^*$-module over $\End_A(E)$.

Finally, to see that~(4) implies~(3), note that associativity of multiplication in $\Mult(\L(E))$ shows that
\begin{gather*}
\xi \cdot (\eta \,|\, \zeta)_{\Mult(A)} = {_{\End_A(E)}(\xi \,|\, \eta)\cdot\zeta} \qquad\text{for all $\xi,\eta,\zeta \in \Mult(E)$.}
\end{gather*}
So for $\xi,\eta \in \Mult(E)$, the action of ${_{\End_A(E)}(\xi \,|\, \eta)}$ is implemented by the generalised compact operator $\Theta_{\xi, \eta} \in \End^0_{\Mult(A)}(\Mult(E))$. Since $\Mult(E)$ is full as a left $\End_A(E)$-module, the identity operator ${\rm Id}_E$ is in the range of ${_{\End_A(E)}(\cdot \,|\, \cdot)}$, and is therefore a compact operator on~$\Mult(E)$. Since $\Id_E \cdot e = e = \Id_{\Mult(E)} e$ for all $e \in \Mult(E)$, we deduce that ${\rm Id}_{\Mult(E)}$ is a~compact operator. So \cite[Theorem~8.1.27]{BleM} implies that $\Mult(E)$ is f\/initely generated.
\end{proof}

\begin{rmk} Suppose that $A$ and $B$ are $C^*$-algebras and $E_A$ and $E_B$ are right-Hilbert modules over $A$ and $B$ respectively that satisfy the equivalent conditions of Theorem~\ref{thm:equiv-one}. Then the external tensor product $(E\ox F)_{A\ox B}$ satisf\/ies the same conditions. For if $E\cong pA^n$ for some projection $p\in M_n(\Mult(A))$ and $F\cong qB^m$ with $q\in M_m(B)$, then $E\ox F\cong (p\ox q)(A^n\ox B^m)$.
\end{rmk}

\begin{rmk} Suppose that $E$ is an $A$-module satisfying the equivalent conditions of Theo\-rem~\ref{thm:equiv-one}, and that $\varphi \colon A \to B$ is a non-degenerate $*$-homomorphism. Then we can write $E = p A^n$ for some $p \in \Mult(A)$. Since $\varphi$ is nondegenerate, it extends to a~homomorphism $\tilde\varphi \colon \Mult(A) \to \Mult(B)$, and we have $E\ox_{\varphi}B\cong \tilde\varphi(p)B^n$. Hence $E\ox_{\varphi}B$ also satisf\/ies the equivalent conditions in Theorem~\ref{thm:equiv-one}.
\end{rmk}

The f\/inite set~$F$, and hence the unitisation~$A_b$ of~$A$, appearing in Theorem~\ref{thm:equiv-one} is not canonical. The set $F$ depends on the choice of f\/inite right basis $\{\xi_j\}$ for~$\Mult(E)$ to which Kasparov's stabilisation theorem is applied in the second paragraph of the proof of Theorem~\ref{thm:equiv-one}. It is less obvious, but also true, that even if two choices of f\/inite frame $\{\xi_j\}$, $\{\xi'_j\}$ yield the same unitisation~$A_b$ of~$A$, the enveloping projective modules $p A_b^n$ and $p'A_b^n$ obtained from these frames may not be isomorphic, as we now
demonstrate.

\begin{Example}\label{ex:hopf-extends} We let $X := \C$. Let $A=C_0(X)$, and let $E=C_0(X)$ the trivial module over $C_0(X)$ with the obvious inner product and multiplication action. Fix any countable locally f\/inite open cover of $X$, f\/ix a partition of unity~$(\varphi_n)_n$ with respect to this cover, and set $e_n=\sqrt{\varphi_n}$ for each $n$. Then $\{e_n\}$ is a frame for~$E$.

The module $\Mult(E)=\Mult(A)$ is a~module over $\Mult(A)=C(\beta X)$, where $\beta X$ denotes the Stone--\v{C}ech compactif\/ication of~$X$. Let us consider two choices of f\/inite frame for this module. The f\/irst choice of frame has a single element
\begin{gather*}
x_1=1_{\Mult(A)}.
\end{gather*}
The construction in the proof of Theorem~\ref{thm:equiv-one} applied to this frame yields $A_b = C(\C \cup \{\infty\}) \cong C(S^2)$, $n = 1$, $p = 1_{A_b}$ and hence $p A_b^n = A_b$, the trivial module over $A_b$.

Now identify $\Mult(E)$ with $\big(\begin{smallmatrix} 1 & 0 \\0 & 0 \end{smallmatrix}\big) \binom{\Mult(A)}{\Mult(A)}$, and consider the elements
$y_i \in \Mult(E)$ given by continuous extension of the functions
\begin{gather*}
y_1=\begin{pmatrix} \frac{1}{\sqrt{1+|z|^2}} \\ 0\end{pmatrix}\qquad\mbox{and}\qquad
y_2=\begin{pmatrix} \frac{\ol{z}}{\sqrt{1+|z|^2}} \\ 0\end{pmatrix}
\end{gather*}
to $\beta X$. Then $(y_2 \,|\, y_2)_{\Mult(A)}$ is identically $1$ on the boundary $\beta X \setminus X$, and the other inner products $(y_i \,|\, y_j)_{\Mult(A)}$ are identically zero on the boundary. So all the $(y_i \,|\, y_j)_{\Mult(A)}$ take values in $C(X \cup \{\infty\})$, and hence the construction in the proof of Theorem~\ref{thm:equiv-one} applied to this frame again yields $A_b = C(\C \cup\{\infty\}) \cong C(S^2)$. Now $n =
2$, and calculation shows that
\begin{gather*}
p = \frac{1}{1+|z|^2}\begin{pmatrix} 1 & \ol{z}\\ z & |z|^2\end{pmatrix}.
\end{gather*}
Def\/ine $w\in M_2(\Mult(A))$ by continuous extension of the function
\begin{gather*}
w(z) = \begin{pmatrix} \frac{1}{\sqrt{1+|z|^2}} & 0\\ \frac{z}{\sqrt{1+|z|^2}} & 0\end{pmatrix}
\end{gather*}
to $\beta X$. We have
\begin{gather*}
ww^*=\frac{1}{1+|z|^2}\begin{pmatrix} 1 & \ol{z}\\ z & |z|^2\end{pmatrix} = p\qquad\mbox{and}\qquad
w^*w=\begin{pmatrix} 1 & 0\\ 0 & 0\end{pmatrix}.
\end{gather*}
So this $w$ def\/ines an isomorphism of modules $pA^2 \cong A$ over $A = C_0(X)$, but this does not extend to an isomorphism of modules over $A_b$ since $w$ is not well-def\/ined on $S^2$. In particular, $p A_b^2$ is the module of sections of the Hopf line bundle over $S^2 \cong \C \cup \{\infty\}$. This is certainly not trivial, and so not isomorphic to the trivial extension obtained for the f\/irst choice of frame above.
\end{Example}

One important situation to which Theorem~\ref{thm:equiv-one} applies is suspensions. Given a bimodule $E$ over an algebra $A$ we can def\/ine the suspended module $\S E$ over the suspension $\S A$ by
\begin{gather*}
\S E=C_0(\R)\ox E,\qquad (f_1\ox a_1)(g\ox e)(f_2\ox a_2) =f_1gf_2\ox a_1ea_2
\end{gather*}
and with the obvious inner product.

\begin{corl}\label{lem:suspensions} Let $A$ be a unital $C^*$-algebra and $E$ a finitely generated Hilbert module over~$A$. Then $\Mult(\S E)$ is full as a left-Hilbert $\End_{C(\beta\R)\ox A}$-module, and so $\S E$ is $C(\beta\R)\ox A$ projective. In fact $\Mult(\S E)$ is $A^\sim$ projective where $A^\sim$ is the minimal unitisation.
\end{corl}

\begin{proof}
Given any frame $\{\xi_j\}$ for $E$, one checks that $\{1\ox\xi_j\}$ is a frame for $C(\beta\R)\ox E$. So the result follows from Theorem~\ref{thm:equiv-one}.
\end{proof}

The fullness hypothesis in Theorem~\ref{thm:equiv-one} is quite restrictive. In particular, the multiplier module of a Hilbert module ${E_B}$ need not be full as a left $\End_A(E)$-module. For example, if $B$ is unital then $\Mult(E) = E$ \cite{ER}, and so if $\End_A^0(E)$ is non-unital then $\Mult(E)$ is not full as a~left $\End_A(E)$-module.

\begin{Example}Let $\H := \ell^2(\N)$ regarded as a $\K(\H)$--$\C$-imprimitivity bimodule, and let $\overline{\H}$ denote the conjugate $\C$--$\K(\H)$-equivalence. Let $A = \K(\H) \oplus \C$. Let $E := \H \oplus \overline{\H}$ and def\/ine an $A$-bi\-module structure on~$E$ by
\begin{gather*}
(T_1,\lambda_1)\big(\xi_1,\ol{\xi_2}\big)(T_2,\lambda_2)=\big(T_1\xi_1 \lambda_2, \lambda_1\ol{\xi_2}T_2\big)
=\big(T_1\xi_1 \lambda_2, \ol{T_2^*\xi_2\ol{\lambda_1}}\big)
\end{gather*}
for all $T_j\in \K(\H)$, $\lambda_j\in \C$, $\xi_j\in\H$. Def\/ine a right inner-product on $E$ by
\begin{gather*}
\big((\xi_1, \ol{\xi_2}) \,|\, (\eta_1, \ol{\eta_2})\big)_{\K(\H) \oplus \C} = \big(\langle \xi_1, \eta_1\rangle,\Theta_{\xi_2, \eta_2}\big).
\end{gather*}
Since $E$ decomposes as a direct sum of imprimitivity bimodules each of which has a unital algebra acting on one side or the other, the remark on pages 295~and~296 of~\cite{ER} shows that $E$ is equal to its multiplier module $\Mult(E)$. The linking algebra is
\begin{gather*}
\begin{pmatrix} A & \H\oplus\ol{\H}\\ \overline{\H}\oplus\H & A\end{pmatrix}
\end{gather*}
with multiplier algebra
\begin{gather*}
\begin{pmatrix} \mathcal{B}(\H)\oplus\C & \H\oplus\ol{\H}\\
\overline{\H}\oplus\H & \mathcal{B}(\H)\oplus\C\end{pmatrix}.
\end{gather*}
Hence $\Mult(E)=E=\H\oplus\ol{\H}$ is not full as a left $\End_A(E)$-module.
\end{Example}

\section{Application to vector bundles}

Recall that for us a vector bundle is always a locally trivial, f\/inite-rank, complex vector bundle equipped with a continuous family of inner products.

\begin{thm}
\label{cor:VB}
Let $X$ be a second-countable locally compact Hausdorff space of finite topological dimension.
\begin{enumerate}\itemsep=0pt
\item[$a)$]\label{it:bundle extension} Suppose that $V\to X$ is a vector bundle of rank $n$. Then there exists a vector bundle $\tilde{V}$ of rank $n$ over the Stone--\v{C}ech compactification $\beta X$ such that $V \cong \tilde{V}|_X$.
\item[$b)$] There is a metrisable compactification $X^c$ and a vector bundle $V^c\to X^c$ such that $V=V^c|_X$.
\end{enumerate}
\end{thm}
\begin{proof}
(a) Let $E=\Gamma_0(X,V)$, and let $A = C_0(X)$. By Theorems~\ref{thm:nonunital-SS} and~\ref{thm:equiv-one}, it suf\/f\/ices to show that $\Mult(E)$ is full as a left $\End_A(E)$-module. Using $\Mult(E)\cong {\rm Hom}_A(A,E)$, we see that every continuous bounded section of $V$ def\/ines an element of $\Mult(E)$.

Fix a countable open cover $\mathcal{U}$ of $X$ by sets on which $V$ is trivial. Since $X$ has f\/inite topological dimension, say $\dim(X) = d$, we can assume, by ref\/ining if necessary, that there is a~partition $\mathcal{U} = \bigsqcup\limits^d_{i=0} \mathcal{U}_i$ such that distinct elements~$U_1$,~$U_2$ of any given $\mathcal{U}_i$ are disjoint. Fix a partition of unity $\{h_U \colon U\in \mathcal{U}\}$ and choose linearly independent sections $\{f_{U,j} \colon U \in \mathcal{U},\ j=1,\dots,n\}$ of $V|_U$ such that
\begin{gather*}
\sum_{j=1}^n{}_{\End_A(E)}(f_{U,j}(x) \,|\, f_{U,j}(x)) =\sum_{j=1}^n\Theta_{f_{U,j}(x) , f_{U,j}(x)} = h_U(x){\rm Id}_V
\end{gather*}
for all $x$, $U$. For each $0 \le i \le d$ and $1\leq j\leq n$, the pointwise sum $F_{i,j} := \sum\limits_{U \in \mathcal{U}_i}
f_{U,j}$ is a bounded section of $V$ and so belongs to $\Mult(E)$. We then have
\begin{gather*}
{\rm Id}_V = \sum^d_{i=1}\sum_{j=1}^n {}_{\End_A(E)}(F_{i,j} \,|\, F_{i,j}),
\end{gather*}
so $\Mult(E)$ is full as a left $\End_A(E)$-module. A similar argument shows that $\Mult(E)$ is a~full right $C(\beta X)$-module.

For~(b), note that since $X$ is second-countable, $C_0(X)$ is separable. Thus by part~(1) of Theorem~\ref{thm:equiv-one} we see that $E$ is $A_b$-f\/inite projective for a separable unitization~$A_b$ of~$C_0(X)$. So $A_b \cong C(X^c)$ for some second-countable, and hence metrisable, compactif\/ication of $X$. The module of sections of the restriction of $\tilde{V}$ to $X^c$ is then a f\/inite projective module over~$C(X^c)$.
\end{proof}

\begin{corl}\label{thm:equiv-two} Let $A$ be nonunital, separable and commutative, so $A\cong C_0(X)$ and suppose that $X$ is of finite topological dimension. Let $E$ be a right $C^*$-$A$-module. Then the conditions $(1)$--$(4)$ of Theorem~{\rm \ref{thm:equiv-one}} are equivalent to
\begin{enumerate}\itemsep=0pt
\item[$5.$] There is a vector bundle $V\to X$ such that $E\cong \Gamma_0(X,V)$.
\end{enumerate}
\end{corl}

\begin{proof}
Suppose $E$ satisf\/ies~(5). By Theorem~\ref{cor:VB}, the vector bundle $V$ extends to a~bundle $\tilde{V}$ on the Stone--\v{C}ech compactif\/ication. By Theorem~\ref{thm:nonunital-SS}, $E=\Gamma_0(X,\tilde{V}|X)\cong p(C(\beta X)^N)$ for some~$N$ and some~$p\in M_N(C(\beta X))$, and so $E$ satisf\/ies~(1). Conversely, if $E$ is $A_b$-f\/inite projective for some unitisation $A_b$, then Theorem~\ref{thm:nonunital-SS} implies that there is a vector bundle $\tilde{V}$ over the spectrum of~$A_b$ such that $E\cong \Gamma_0(X,\tilde{V}|_X)$. This completes the proof.
\end{proof}

\begin{corl} Suppose that $X$ is a locally compact Hausdorff space with finite topological dimension. Then a right $C^*$-$C_0(X)$-module $E$ is $C(X^c)$ finite projective for some compactification~$X^c$ of~$X$ if and only if $E\cong\Gamma_0(X,V)$ for some vector bundle~$V\to X$.
\end{corl}

\section{Multiplier modules of bi-Hilbertian bimodules}

In this section, we investigate the structure of the multiplier module of a bi-Hilbertian bimodule in the sense of Kajiwara--Pinzari--Watatani (see Def\/inition~\ref{defn:bimod} below). We show that if $E$ has f\/inite Watatani index then its multiplier module can always be made into a~bi-Hilbertian bimodule with f\/inite numerical index in its own right.

We then show that the passage of f\/inite Watatani index from~$E$ to $\Mult(E)$ is much less common; it is equivalent, for example, to fullness of $\Mult(E)$ as a left-Hilbert $\End_{A}(E)$-module. In particular, f\/inite Watatani index does not help to weaken the fullness hypotheses of Theorem~\ref{thm:equiv-one}.

\begin{defn}[{\cite[Definition~2.3]{KajPinWat}}] \label{defn:bimod} Let $A$ be a $\sigma$-unital $C^*$-algebra. A \emph{bi-Hilbertian $A$-bimodule} is a countably generated full right Hilbert $C^*$-$A$-module with inner product $(\cdot \,|\, \cdot)_A$ which is also a countably generated full left Hilbert $C^*$-$A$-module with inner product ${_A(\cdot \,|\, \cdot)}$ such that the left action of $A$ is adjointable with respect to $(\cdot \,|\, \cdot)_A$ and the right action of $A$ is adjointable with respect to ${_A(\cdot \,|\, \cdot)}$.
\end{defn}

Since a bi-Hilbertian bimodule is complete in the norms coming from both the left and the right inner products, these two norms are equivalent.

Let $E$ be a bi-Hilbertian $A$-bimodule. We recall from \cite[Def\/inition~2.8]{KajPinWat} the def\/inition of the right numerical index of $E$. We say that $E$ has \emph{finite right numerical index} if there exists $\lambda > 0$ such that
\begin{gather*}\textstyle
\big\|\sum_i {_A(f_i \,|\, f_i)}\big\| \le \lambda \big\|\sum_i \Theta_{f_i, f_i}\big\|
 \qquad\text{for all $n$ and all $f_1, \dots, f_n \in E$.}
\end{gather*}
The \emph{right numerical index} of $E$ is the inf\/imum of the numbers $\lambda$ satisfying the above inequality.

Let $E$ be a bi-Hilbertian $A$-bimodule, countably generated as a right module and with f\/inite right numerical index. By \cite[Corollaries 2.24~and~2.28]{KajPinWat}, the left action of $A$ on $E$ is by compact endomorphisms with respect to $(\cdot \,|\, \cdot)_A$ if and only if, for every frame $\{e_j\}$ for $(E, (\cdot \,|\, \cdot)_A)$, the series
\begin{gather*}
\sum_{j\geq 1}{}_A(e_j \,|\, e_j)
\end{gather*}
converges strictly in $\Mult(A)$. In this case we denote the strict limit by $\rInd(E)$. We note that $\rInd(E)$ is independent of the frame $\{e_j\}$, and is called the \emph{right Watatani index} of $A$. This $\rInd(E)$ is a positive central element of~$\Mult(A)$, and is invertible if and only if the left action of $A$ is implemented by an injective homomorphism $A \to \End_A^0(E)$.

As in \cite{KajPinWat}, we simply say that $E$ has \emph{finite right Watatani index} if it has f\/inite right numerical index and the left action is by compacts.

\begin{thm}\label{thm:main} Let $A$ be a nonunital $\sigma$-unital $C^*$-algebra and let $E$ be a bi-Hilbertian $A$-bimodule, countably generated as a right module and with injective left action. Suppose that $E$ has finite right Watatani index. Then ${_A(\cdot \,|\, \cdot)}$ extends to a strictly continuous $\Mult(A)$-valued inner product on~$\Mult(E)$, with respect to which $\Mult(E)$ is a bi-Hilbertian $\Mult(A)$-bimodule with finite right numerical index.
\end{thm}
\begin{proof}
When the bi-Hilbertian $A$-bimodule $E$ has f\/inite right Watatani index, \cite[Corollary~2.11]{KajPinWat} shows that there is a~positive $A$-bilinear norm continuous map $\Phi\colon \End_A^0(E)\to A$ such that
\begin{gather*}
{}_A(e\,|\, f)=\Phi(\Theta_{e,f}),\qquad e,\,f\in E.
\end{gather*}
Proposition~2.27 of \cite{KajPinWat} implies that $\Phi$ extends to a bounded strictly continuous positive $A$-bilinear map $\overline{\Phi} \colon \End_{A}(E)\to \Mult(A)$ with $\|\ol{\Phi}\| \le \|\rInd(E)\|$.

Now let $\Mult(E)$ be the multiplier module of $E$. The inclusion $A \hookrightarrow \End^0_A(E) \subseteq \L(E)$ implementing the left action extends to a homomorphism $\Mult(A) \to \End_A(E) \subseteq \Mult(\L(E))$, giving a left action of $\Mult(A)$ on $\Mult(E)$ by $(\cdot \,|\, \cdot)_{\Mult(A)}$-adjointable operators.

We saw at~\eqref{eq:E multE inv} that for $e,f \in \Mult(E)$ the operator $\Theta_{e,f} \in \End^0_{\Mult(A)}(\Mult(E))$ restricts to an adjointable operator on $E$, so we can def\/ine ${_{\Mult(A)}(\cdot \,|\, \cdot)} \colon \Mult(E) \times \Mult(E) \to \Mult(A)$ by
\begin{gather*}
{}_{\Mult(A)}(e\,|\, f):=\ol{\Phi}(\Theta_{e,f}).
\end{gather*}
By def\/inition of $\Phi$, this strictly continuous form extends ${_A(\cdot \,|\, \cdot)}$.

We claim that this is a left-$\Mult(A)$-linear $\Mult(A)$-valued inner-product on~$\Mult(E)$. To see this, observe that sesquilinearity of $(e,f) \mapsto \Theta_{e,f}$, linearity of restriction and linearity of $\overline{\Phi}$ show that ${_{\Mult(A)}(\cdot \,|\, \cdot)}$ is sesquilinear. For $a \in \Mult(A)$ and $e,f \in \Mult(E)$, we use the $A$-bilinearity of $\overline{\Phi}$ to see that
\begin{gather*}
a \cdot {_{\Mult(A)}(e \,|\, f)} = a \overline{\Phi}(\Theta_{e,f}|_E) = \overline{\Phi}(a \Theta_{e,f}|_E) = \overline{\Phi}(\Theta_{a \cdot e, f}|_E) = {_{\Mult(A)}(a \cdot e \,|\, f)},
\end{gather*}
so ${_{\Mult(A)}(\cdot \,|\, \cdot)}$ is left-$A$-linear. It is positive because
$\overline{\Phi}$ is.

By \cite[Proposition 2.27(1)]{KajPinWat}, the maximal $\lambda' > 0$ such that $\lambda'\Vert (e\,|\, e)_A\Vert\leq \Vert{}_A(e\,|\, e)\Vert$ for all $e\in E$ satisf\/ies
\begin{gather*}
\lambda'T\leq \ol{\Phi}(T)\leq \Vert T\Vert \rInd(E)\qquad\text{for all $0\leq T\in \End_A(E)$.} 
\end{gather*}
Hence ${_{\Mult(A)}(\cdot \,|\, \cdot)}$ is positive def\/inite. The same inequality combined with the norm equality $\Vert \Theta_{e,e}\Vert=\Vert (e\,|\, e)_{\Mult(A)}\Vert$ gives
\begin{gather*}
\lambda'\Vert (e\,|\, e)_{\Mult(A)}\Vert \leq\Vert {}_{\Mult(A)}(e\,|\, e)\Vert \leq \Vert \rInd(E)\Vert\,\Vert (e\,|\, e)_{\Mult(A)}\Vert,
\end{gather*}
and so the norms on $\Mult(E)$ coming from the left and right inner products are equivalent.

To see that $\Mult(E)$ has f\/inite right numerical index, let $g_1, \dots, g_m \in \Mult(E)$ be any f\/inite set. Since $\overline{\Phi}$ is bounded we have
\begin{gather*}
\textstyle \big\|\sum_i {_{\Mult(A)}(g_i \,|\, g_i)}\big\| = \big\|\sum_i \overline{\Phi}(\Theta_{g_i, g_i})\big\| \le \Vert\ol{\Phi}\Vert\, \big\|\sum_i \Theta_{g_i, g_i}\big\|.
\end{gather*}
In particular, given a frame $\{g_j \colon j \in \N\}$ for $\Mult(E)$ and any f\/inite subset
$I$ of $\N$, we have
\begin{gather*}\textstyle
\big\|\sum_{i\in I} {_{\Mult(A)}(g_i \,|\, g_i)}\big\| \leq \Vert\ol{\Phi}\Vert.
\end{gather*}
So $\Mult(E)$ has f\/inite right numerical index.
\end{proof}

Our next theorem shows that while f\/inite right numerical index passes easily from $E$ to $\Mult(E)$ as above, f\/inite right Watatani index is another question.

\begin{thm}\label{thm:TFAE} Let $A$ be a $\sigma$-unital non-unital $C^*$-algebra and let $E$ be a~bi-Hilbertian $A$-bi\-module, countably generated as a right module and with injective left action. Suppose that~$E$ has finite right Watatani index. Then the following are equivalent:
\begin{enumerate}\itemsep=0pt
\item[$1)$] $\Mult(E)$ is full as a left $\Mult(A)$-module;
\item[$2)$] ${\rm Id}_{\Mult(E)}$ is a compact endomorphism of $\Mult(E)_{\Mult(A)}$;
\item[$3)$] the left action of $\Mult(A)$ on $\Mult(E)_{\Mult(A)}$ is by compact endomorphisms;
\item[$4)$] $\rInd(\Mult(E))\in\Mult(A)$ $($that is, $\Mult(E)$ has finite right Watatani index$)$; and
\item[$5)$] $E$ is $\Mult(A)$-finite projective as in Definition~{\rm \ref{def:a-b-fgp}}.
\end{enumerate}
\end{thm}

\begin{proof} The left action of $\Mult(A)$ is injective. Hence \cite[Corollaries~2.20 and~2.28]{KajPinWat} gives \mbox{(1)\;$\Leftrightarrow$\;(4)}. We have
\mbox{(2)\;$\Leftrightarrow$\;(3)} because $\End_{\Mult(A)}^0(\Mult(E))$ is an ideal of $\End_{\Mult(A)}(\Mult(E))$ and because the identity of $\Mult(A)$ acts as ${\rm Id}_{\Mult(E)}$. Similarly, \mbox{(5)\;$\Rightarrow$\;(3)} since in this case all the endomorphisms are compact. The implication \mbox{(1)\;$\Rightarrow$\;(5)} follows from the implication \mbox{(4)\;$\Rightarrow$\;(2)} in Theorem~\ref{thm:equiv-one}. So it suf\/f\/ices to prove that \mbox{(3)\;$\Leftrightarrow$\;(4)}, which follows from \cite[Theorem~2.22]{KajPinWat}.
\end{proof}

\begin{rmk}If the equivalent conditions (1)--(5) of Theorem~\ref{thm:TFAE} hold, then Theorem~\ref{thm:equiv-one} shows that in fact $E$ is $A_b$-f\/inite projective for some subalgebra of $\Mult(A)$ generated by $A$ and just f\/initely many additional elements.
\end{rmk}

\begin{rmk} By Theorem~\ref{thm:main}, if the equivalent conditions in Theorem~\ref{thm:TFAE} hold, then the left inner product on $E$ extends by strict continuity to a $\Mult(A)$-valued left inner product on $\Mult(E)$ under which $\Mult(E)$ itself becomes a bi-Hilbertian bimodule with invertible f\/inite right Watatani index.
\end{rmk}

\begin{Example}We describe a concrete example where Theorem~\ref{thm:main} applies but the equivalent conditions in Theorem~\ref{thm:TFAE} do not hold. Consider the non-unital $C^*$-algebra $A := \K$, the compact operators on $\ell^2(\N)$. Let $E$ be the external tensor product $E = \ell^2 \otimes \K$, which is a Morita equivalence from $\K \cong \K \otimes \K$ to $\C \otimes \K \cong \K$, and hence an $A$--$A$-imprimitivity bimodule. So $E$ has f\/inite right Watatani index $\rInd(E) = 1_{M(\K)}$.

It is routine to check that $\Mult(E) \cong \ell^2 \otimes \mathcal{B}(\ell^2)$. Since the left action of $\Mult(A) = \mathcal{B}(\ell^2)$ on this module is not by compact operators, Theo\-rem~\ref{thm:TFAE}(2) fails, and therefore so do the other conditions. In particular, while $\Mult(E)$ has f\/inite numerical index by Theorem~\ref{thm:main}, it does not have f\/inite right Watatani index.
\end{Example}

Finite right Watatani index can nevertheless be a useful tool for deciding when we obtain f\/inite projective modules over unitisations. For a commutative algebra~$A$, we can equip a right inner product module $E$ over $A$ with a left action and inner product via the formulae
\begin{gather}\label{eq:comm biHilb structure}
a\cdot e:=ea,\qquad\mbox{and}\qquad {}_A(e_1 \,|\, e_2):=(e_2 \,|\, e_1)_A,\qquad\text{for all} \ \ e,\,e_1,\,e_2 \in E,\ a\in A.
\end{gather}
So we are in the bi-Hilbertian setting. For each character $\phi \in \widehat{A}$, the (completion of the) quotient $\H_\phi := E/(E \cdot \ker\phi)$ is a Hilbert space with inner product $\langle e_1,e_2\rangle = \phi((e_1\,|\, e_2)_A)$.

\begin{corl}\label{cor:VB2} Let $X$ be a second-countable locally compact Hausdorff space of finite topological dimension, and let $A=C_0(X)$. Let $E$ be a full right Hilbert $A$-module, regarded as a bi-Hilbertian bimodule as in~\eqref{eq:comm biHilb structure}. The following are equivalent:
\begin{enumerate}\itemsep=0pt
\item[$1)$] The map $\phi \mapsto \dim \H_\phi$ is a continuous bounded function from $X$ to $[0,\infty)$;
\item[$2)$] $E$ is isomorphic to $\Gamma_0(X,V)$ for some $($finite rank$)$ vector bundle $V\to X$;
\item[$3)$] $E$ has finite right Watatani index.
\end{enumerate}
\end{corl}
\begin{proof}
We f\/irst prove that~(1) implies~(2). For $e\in E$, def\/ine a section $S_e \colon X \to \bigcup_{\phi \in X} \H_\phi = E/(E \cdot \ker\phi)$ by $S_e(\phi) = e + E\cdot\ker\phi$. This set of sections is a vector space and is f\/ibrewise dense (in fact, f\/ibrewise surjective), and so by \cite[Theorem~II.13.18]{FellDoran} there is a unique topology on $V := \bigcup_\phi \H_\phi$ under which these sections are continuous.

For $a \in C_0(X)$ and $e \in E$, we see that $S_{e \cdot a}(\phi) = a(\phi)S_e(\phi)$ for all $\phi\in \hat{A}$. So \cite[Corollary~II.14.7]{FellDoran} shows that the $S_e$ are uniformly-on-compacta dense in $\Gamma_0(X, V)$. A simple argument using completeness
of $E$ then shows that $e \mapsto S_e$ is a surjection of $E$ onto $\Gamma_0(X, V)$. By def\/inition, this surjection carries the right action and inner product on $E$ to those on $\Gamma_0(X, V)$. So $e \mapsto S_e$ is an isomorphism $E \cong \Gamma_0(X, V)$.

We must check that $V$ is locally trivial. Remark~II.13.9 of \cite{FellDoran} states that f\/inite-rank continuous bundles of Banach spaces of constant dimension are locally trivial; we provide a proof in our setting for completeness (see also \cite[Remarque, p.~231]{DD}).

Fix $\phi \in X$, and choose an orthonormal basis $\{e_1, \dots, e_n\}$ for $\H_\phi$. Fix elements $\xi_i \in E$ such that $\xi_i + E \cdot \ker\phi = e_i$. By scaling by an appropriate element of $C_0(X)$ we can assume that there is a neighbourhood $U$ of $\phi$ such that for $\psi \in U$, the element $e^\psi_i := \xi_i + E \cdot \ker\psi$ satisf\/ies $\|e^\psi_i\| = 1$. Since $\psi \mapsto \dim\H_\psi$ is continuous, and since $(\xi_i \,|\, \xi_j)$ is continuous for all $i$, $j$, by shrinking $U$ if necessary we can assume that $\dim\H_\psi = n$ for all $\psi \in U$ and that $|\langle e^\psi_i, e^\psi_j\rangle| < \frac{1}{2n^2}$ for $i \not= j$ and $\psi \in U$. Now if $\psi \in U$ and $\sum_i \alpha_i e^\psi_i = 0$, then
\begin{gather*}
0 = \Big|\Big\langle \sum_i \alpha_i e^\psi_i, \sum_j \alpha_j e^\psi_j\Big\rangle\Big|
 \ge \sum_i |\alpha_i|^2 \|e^\psi_i\|^2 - \frac{(n^2-n) \max_i|\alpha_i|^2}{2n^2} \ge \max_i|\alpha_i|^2/2.
\end{gather*}
So $V$ admits a continuous choice of basis on $U$, so is locally trivial.

To see that~(2) implies~(3), suppose that $E$ has the form $\Gamma_0(X, V)$. Choose a~locally f\/inite cover of $X$ by sets $U_i$ on which on which $V$ is trivial, say $V|_{U_i} \cong U_i \times \C^{n_i}$. Since $V$ is f\/inite rank, $\sup_i n_i<\infty$. Let $e_1,\dots,e_j$ denote the bounded sections of $V|_{U_i}$ given by the standard basis of~$\C^{n_i}$.

Choose a partition of unity $(h_i)$ subordinate to the $U_i$, and put $f_{i,j} = \sqrt{h_i}e_j$ for $1\leq j \le n_i$ and $i \in \N$. Then just as in Theorem~\ref{cor:VB} we see that
\begin{gather*}
\sum_{i,j}\Theta_{f_{i,j},f_{i,j}}=\sum_ih_i{\rm Id}_{V|_{U_i}}={\rm Id}_V.
\end{gather*}
Using the module structure def\/ined in \eqref{eq:comm biHilb structure}, we see that the left inner product of two sections $e=\sum_jc_je_j$, $f=\sum_kd_ke_k$ supported in $U_i$ is given by
\begin{gather*}
{}_A(e\,|\, f)= \sum_{j,k}c_j\ol{d_k}(e_k\,|\, e_j)_A= \sum_jc_j\ol{d_j}=
 \sum_{j}(e_j\,|\, e)_A(f\,|\, e_j)_A= \sum_j(e_j\,|\,\Theta_{e,f}e_j)_A\\
 \hphantom{{}_A(e\,|\, f)}{} = {\rm Trace}(\Theta_{e,f}).
\end{gather*}
Hence for any f\/inite set $\{e_k \colon k \le K\} \subseteq E$, we have
\begin{gather*}
\Big\Vert \sum_k{}_A(e_k\,|\, e_k)\Big\Vert \leq (\sup_i n_i)\Big\Vert \sum_k\Theta_{e_k,e_k}\Big\Vert,
\end{gather*}
and so $E$ has f\/inite numerical index.

This also shows that $\sum_j{_A(f_{i,j} \,|\, f_{i,j})}=h_in_i$ on the open set $U_i$, and so $\big(\sum_{i,j} {_A(f_{i,j} \,|\, f_{i,j})} \big)a$ converges to $n_i a$ for $a \in C_c(U_i)$. So $\sum_{i,j} {_A(f_{i,j} \,|\, f_{i,j})}$ converges strictly. That is, $E$ has f\/inite right Watatani index.

Finally, for~(3) implies~(1), suppose that $E$ has f\/inite right Watatani index. By def\/inition of the bi-Hilbertian structure~\eqref{eq:comm biHilb structure}, the right Watatani index of $E$ is the function $\phi \mapsto \dim\H_\phi$. By def\/inition of f\/inite Watatani index, we deduce that $(\phi \mapsto \dim\H_\phi) \in \Mult(C_0(X)) \cong C_b(X)$, giving~(1).
\end{proof}

\subsection*{Acknowledgements} This research was supported by Australian Research Council grant DP150101595. It was motivated by questions arising in projects with our collaborators Francesca Arici, Magnus Gof\/feng, Bram Mesland and Dave Robertson, and we thank them for all that we have learned from them. We are very grateful to the anonymous referee who read the manuscript very closely and made numerous very helpful suggestions that have signif\/icantly strengthened our results and streamlined our proofs. Thanks, whoever you are.

\pdfbookmark[1]{References}{ref}
\LastPageEnding

\end{document}